\documentclass{amsart}

\usepackage{latexsym}
\usepackage{amssymb}
\usepackage{amsmath,mathabx,mathtools}
\usepackage{natbib}

\newtheorem{theorem}{Theorem}[section]
\newtheorem{lemma}[theorem]{Lemma}
\newtheorem{proposition}[theorem]{Proposition}
\newtheorem{corollary}[theorem]{Corollary}

\theoremstyle{definition}
\newtheorem{definition}[theorem]{Definition}

\theoremstyle{remark}
\newtheorem{remark}[theorem]{Remark}

\numberwithin{equation}{section}

\newcommand{\C}{\mathbb{C}}
\newcommand{\R}{\mathbb{R}}

\newcommand{\E}{\mathcal{E}}

\newcommand{\f}{\varphi}

\newcommand{\Aut}{\operatorname{Aut}}
\newcommand{\WAP}{\operatorname{WAP}}
\newcommand{\wap}{\operatorname{wap}}

\newcommand{\Int}{\operatorname{Int}}

\newcommand{\Om}{\Omega}

\newcommand{\om}{\omega}

\newcommand{\la}{\langle}
\newcommand{\ra}{\rangle}
\newcommand{\Cb}{\operatorname{C_b}}
\newcommand{\Cbr}{\operatorname{C_{b,r}}}
\newcommand{\ol}{\overline}
\newcommand{\mk}{\mathfrak m}

\begin{document}

\title[$G$-finite vNa]{$G$-finite von Neumann algebras and weakly almost periodic functions}

\author{Paul Jolissaint}
\address{Universit\'e de Neuch\^atel,
       Institut de Math\'ematiques,       
       E.-Argand 11,
       2000 Neuch\^atel, Switzerland}
       
\email{paul.jolissaint@ik.me}


\date{\today}


\begin{abstract}
Let $M$ be a von Neumann algebra which acts in a standard way on the Hilbert space $H$, let $G$ be a subgroup of the group of all $*$-automorphisms of $M$. We prove that $M$ is $G$-finite (in the sense of I. Kov\'acs and Sz\"ucs, i.e. the set of all normal, $G$-invariant states on $M$ is separating) if and only if, for every $x\in M$ and all $\xi,\eta\in H$, the coefficient function $g\mapsto \la g(x)\xi|\eta\ra$ is weakly almost periodic.
\end{abstract} 

\maketitle

\section{Introduction}

Let $M$ be a von Neumann algebra, let $M_*$ be its predual, let $\Aut(M)$ be the group of all $*$-automorphisms of $M$; we consider here exclusively the action of $M$ on its standard form $(M,H,J,P)$ as in \citep{haa} and in \citep[Definition X.1.3]{Tak2}: $H$ is a Hilbert space with $M\subset B(H)$, $J$ is a conjugate linear, isometric involution on $H$, $P$ is a selfdual cone in $H$ (i.e. $P=\{\eta\in H\colon \la\eta|\xi\ra\geq 0\ \forall \xi\in P\}$) with the properties
\begin{enumerate}
\item $JMJ=M'$, the commutant of $M$ in $B(H)$;
\item $JcJ=c^*$ for every $c\in Z(M)$, the center of $M$;
\item $J\xi=\xi$ for every $\xi\in P$;
\item $aJaJ(P)\subset P$ for every $a\in M$.
\end{enumerate}
It follows from uniqueness of the standard form of $M$ that the group $\Aut(M)$ admits a unitary implementation, i.e. there is a unique unitary representation $u:\Aut(M)\rightarrow U(H)$ such that $u(g)xu(g^{-1})=g(x)$ for all $g\in \Aut(M), x\in M$ and such that $J=u(g)Ju(g^{-1}), u(g)P=P$ for every $g\in \Aut(M)$. Moreover, $u$ is a homeomorphism of $\Aut(M)$ onto a closed subgroup of $U(H)$ where the first is equipped with the u-topology (generated by the seminorms $\Phi\mapsto \|\f\circ\Phi\|, \f\in M_*$) and the latter with the strong ( = weak) operator topology. See for instance \citep[Section 3]{haa}.

\medskip
Given a subgroup $G$ of $\Aut(M)$, I. Kov\'acs and J. Sz\"ucs introduced the notion of $G$\textit{-finite} von Neumann algebras in \citep{KS}, which means that the set $S_*^G(M)$ of all normal, $G$-invariant states on $M$ separates the points of $M$. They proved among others that this is equivalent to the existence of a suitable conditional expectation $\E^G$ from $M$ onto the fixed point subalgebra $M^G=\{x\in M\colon g(x)=x\ \forall g\in G\}$. See the more precise statement in the next theorem below. It is worth mentioning that other characterizations were added, namely by E. St\o rmer in \citep{Stoermer}.

Before stating these characterizations, we need to recall some facts and fix notation. For every $x\in M$, let $K_G(x)$ be the $\sigma$-weakly closed convex hull of the orbit $\{g(x)\colon g\in G\}$ of $x$ and, for $\psi\in M_*$, set $G(\psi)=\{\psi\circ g\colon g\in G\}$.

Finally, let $B(M)$ (resp. $B_*(M)$) be the Banach space of all linear, bounded maps $\Phi$ on $M$ (resp. linear, bounded, normal maps on $M$). It turns out that $B(M)$ is a dual Banach space, more precisely, it is the dual of the projective tensor product $M\otimes_\pi M_*$ with respect to the duality bilinear form
\[
\la x\otimes \f,\Phi\ra=\f(\Phi(x))\quad (x\in M, \f\in M_*,\Phi\in B(M)).
\]
Thus, this means that $B(M)$ is equipped with a natural weak$^*$ topology. For details, see for instance \citep[Theorem IV.2.3]{Tak1}.

\medskip
Here are the promised characterizations of $G$-finiteness.

\begin{theorem}\label{thm1.1}
Let $M$ be a von Neumann algebra and let $G$ be a subgroup of $\Aut(M)$. Then the following conditions are equivalent:
\begin{enumerate}
\item $M$ is $G$-finite: if $x\in M_+$ is such that $\f(x)=0$ for every $\f\in S_*^G(M)$, then $x=0$.
\item For every $x\in M$, the set $K_G(x)\cap M^G$ contains exactly one element denoted by $\E^G(x)$; the map $\E^G:M\rightarrow M^G$ is a normal, faithful conditional expectation which is $G$-invariant (i.e. $\E^G\circ g=\E^G$ for every $g\in G$).
\item For every $\psi\in M_*$, its orbit $G(\psi)$ is relatively $\sigma(M_*,M)$-compact.
\item The $w^*$-closure of $G$ in $B(M)$ is contained in $B_*(M)$.
\end{enumerate}
Moreover, if these conditions are satisfied, then $\E^G$ is unique.
\end{theorem}
Equivalence between (1) and (2) is due to I. Kov\'acs and J. Sz\"ucs \citep{KS}, and equivalences between (1), (3) and (4) are due to E. St\o rmer \citep{Stoermer}. 

 Observe that if $G=\Int(M)$ is the group of all inner automorphisms of $M$, then $M$ is $\Int(M)$-finite if and only if it is a finite von Neumann algebra. Apparently, the latter case inspired the authors of \citep{KS}. E. St\o rmer was inspired by the article of F.J. Yeadon \citep{Yeadon} who proved the existence of the canonical center-valued trace on a finite von Neumann algebra as a consequence of Ryll-Nardzewski fixed point theorem.

\medskip
It turns out that one of the main results of our article \citep[Theorem 3.1]{Jol} is a characterization of finite von Neumann algebras in terms of suitable almost weakly periodic functions on the unitary group $U(M)$. Thus, the aim of the present note is to generalize the latter theorem to the more general case of arbitrary subgroups of $\Aut(M)$. Here is our main result; the precise definition of weakly almost periodic functions is recalled in the next section.

\begin{theorem}\label{mainThm}
Let $M$ be a von Neumann algebra which acts standardly on $H$ and $G$ be a subgroup of $\Aut(M)$. Then $M$ is $G$-finite if and only if, for every $x\in M$ and all $\xi,\eta\in H$, the coefficient function $g\mapsto \la g(x)\xi|\eta\ra$ is weakly almost periodic on $G$.
\end{theorem}

\section{Weakly almost periodic coefficient functions}

The proof of Theorem 1.2 requires some preparation on weakly almost periodic functions on topological groups and their properties. 

Thus, let us recall some definitions and fix notation.

\medskip

First, let $G$ be a topological group. We denote by $\Cb(G)$ the $C^*$-algebra of all bounded, continuous, complex-valued functions on $G$ equipped with the uniform norm $\Vert f\Vert_\infty:=\sup_{s\in G}|f(s)|$. 
For $g\in G$ and $f:G\rightarrow\C$, we denote by $g\cdot f:G\rightarrow\C$ (resp. $f\cdot g$) the left (resp. right) translate of $f$ by $g$, i.e.
\[
(g\cdot f)(s)=f(g^{-1}s)\quad\textrm{and}\quad
(f\cdot g)(s)=f(sg)
\]
for all $f:G\rightarrow\C$ and $g,s\in G$. The corresponding left (resp. right) orbit is denoted by $Gf$ (resp. $fG$). A function $f\in \Cb(G)$ is \textit{right uniformly continuous}\footnote{Our definition follows that of P. de la Harpe \citep{HarpeMoy}, contrary to the \cite{EymardMoy}.} if $\Vert g\cdot f-f\Vert_\infty\to 0$ as $g\to 1$. The subset of all right uniformly continuous functions $f\in \Cb(G)$ is a sub-$C^*$-algebra of $\Cb(G)$ denoted by $\Cbr(G)$, and it contains all right translates of all its elements.

A function $f\in \Cbr(G)$ is \textit{weakly almost periodic} if its orbit $Gf$ is weakly relatively compact in $\Cbr(G)$, or equivalently in $\Cb(G)$. It follows from \citep[Proposition 7]{Grothen} that $Gf$ is weakly relatively compact if and only if $fG$ is. The set of all weakly almost periodic functions on $G$ is denoted by $\WAP(G)$; it is a sub-$C^*$-algebra of $\Cbr(G)$, and its main feature, which will play an important role here, is the existence of a unique left and right $G$-invariant mean $\mk_G$ on $\WAP(G)$. One finds a proof of this result for locally compact groups for instance in F.P. Greenleaf's monograph \citep{Greenleaf}, but the proof, based on Ryll-Nardzewski fixed point theorem, works without any change for arbitrary topological groups. See also \citep[Theorem 5.5]{Jol}.

\medskip
Let us state three propositions which yield criteria for weakly almost periodic functions and weak relative compactness of bounded subsets of $\Cb(G)$ or $\Cbr(G)$. The first two are taken essentially from \citep{Grothen} and the third one seems to be new.

\begin{proposition}\label{Gro1}
(\citep[Th\'eor\`eme 6]{Grothen})
Let $G$ be a topological group and let $A\subset \Cb(G)$ be a bounded set. It is weakly relatively compact if and only if, for all sequences $(x_i)\subset G$ and $(f_j)\subset A$ such that the following two limits exist
\[
\ell_1=\lim_i\Big(\lim_j f_j(x_i)\Big)\quad \textrm{and}\quad \ell_2=\lim_j\Big(\lim_i f_j(x_i)\Big)
\]
then $\ell_1=\ell_2$.
\end{proposition}

\begin{corollary}\label{Gro2}
(\citep[Proposition 7]{Grothen}) Let $G$ be as in the previous proposition and let $f\in\Cbr(G)$. Then $f\in \WAP(G)$ if and only if, for all sequences $(x_i),(y_j)\subset G$ such that the following two limits exist
\[
\ell_1=\lim_i\Big(\lim_j f(x_iy_j)\Big)\quad \textrm{and}\quad \ell_2=\lim_j\Big(\lim_i f(x_iy_j)\Big)
\]
then $\ell_1=\ell_2$.
\end{corollary}

\begin{proposition}\label{relcpt}
Let $G$ be a topological group and let $(f_j)_{j\geq 1}\subset \WAP(G)$ be a sequence that has the following properties:
\begin{enumerate}
\item there exists $C>0$ such that $0\leq f_j\leq C$ for every $j$;
\item one has $f_j\geq f_{j+1}$ for every $j$.
\end{enumerate}
Then the set $A=\bigcup_{j\geq 1}Gf_j=\{g\cdot f_j\colon j\geq 1, g\in G\}$ is weakly relatively compact in $\Cbr(G)$.
\end{proposition}
\begin{proof}
Let $\Om$ be the Gelfand spectrum of $\Cbr(G)$, so that we identify $C(\Om)$ with $\Cbr(G)$. Since $A$ is bounded, Eberlein-Smulian Theorem \citep[Theorem A.12]{Glas} implies that it suffices to prove that, for every sequence $(g_n)_{n\geq 1}\subset G$, there exists a subsequence $(g_{n_k})_{k\geq 1}\subset (g_n)$ and $h\in C(\Om)$ such that
\[
\lim_{k\to\infty} g_{n_k}\cdot f_k(\omega)=h(\omega)\quad (\om\in \Om).
\]
As $Gf_1$ is weakly relatively compact, there exist a subsequence $(g_{n_k^{(1)}})\subset (g_n)$ and $h_1\in C(\Om)$ such that
\[
\lim_{k\to\infty}g_{n_k^{(1)}}\cdot f_1(\om)=h_1(\om)\quad (\om\in \Om).
\]
Assumption (1) implies that $0\leq h_1\leq C$. 

Next, as $Gf_2$ is weakly relatively compact, there exist a subsequence $(g_{n_k^{(2)}})\subset (g_{n_k^{(1)}})$ and $h_2\in C(\Om)$ such that 
\[
\lim_{k\to\infty}g_{n_k^{(2)}}\cdot f_2(\om)=h_2(\om)\quad (\om\in \Om).
\]
As $f_2\leq f_1$ and $\lim_{k\to\infty} g_{n_k^{(2)}}\cdot f_1=h_1$, one has $h_2\leq h_1$.

By induction, one finds subsequences $(g_{n_k^{(j)}})$ of $(g_n)$ and functions $h_j\in C(\Om)$ with the following properties:
\begin{itemize}
\item $(g_{n_k^{(j+1)}})\subset (g_{n_k^{(j)}})$ for every $j$;
\item $\lim_{k\to\infty} g_{n_k^{(j)}}\cdot f_{j}(\om)=h_{j}(\om)$ for all $j, \om\in\Om$;
\item $0\leq h_{j+1}\leq h_j\leq C$ for every $j$.
\end{itemize}
Set then $g_{n_k}=g_{n_k^{(k)}}$ for every $k$ and let $h=\inf_j h_j$. Then $h\in C(\Om)$ by Dini's Theorem and we have
\[
\lim_{k\to\infty} g_{n_k}\cdot f_k(\om)=h(\om)
\]
for every $\om\in\Om$.
\end{proof}

Let now $M$ be a von Neumann algebra, let $(M,H,J,P)$ be its standard form and let also $u:\Aut(M)\rightarrow U(H)$ be the canonical implementation of $\Aut(M)$ as described in Section 1. 

We henceforth fix a subgroup $G$ of $\Aut(M)$ equipped with the induced u-topology. For short, we will write $g(T)=u(g)Tu(g^{-1})$ for all $T\in B(H),g\in G$ whenever convenient. 
For $T\in B(H), \xi,\eta\in H$, the associated \textit{coefficient function} is denoted by $\xi*T*\eta$ and it is defined by
\[
\xi*T*\eta(h)\coloneqq \la h(T)\xi|\eta\ra\quad (h\in G).
\]
It is straightforward to verify that every such coefficient function belongs to $\Cbr(G)$, and that it is a constant function if and only if $T\in u(G)'$. See \citep[Lemma 2.2]{Jol}.

As in \citep[Lemma 2.1]{Jol}, the coefficient functions satisfy the properties stated in the next lemma; their proofs are similar to that of the quoted lemma, so that we omit them.

\begin{lemma}\label{coefficients}
Let $T\in B(H),\xi,\eta\in H, g\in G$. Then the following formulas hold:
\begin{equation}\label{2.1}
\xi* T^**\eta=\ol{\eta* T* \xi},
\end{equation}
\begin{equation}\label{2.2}
g\cdot(\xi* T* \eta)=(u(g)\xi)* T*(u(g)\eta),
\end{equation}
\begin{equation}\label{2.3}
(\xi* T*\eta)\cdot g=\xi* g(T)*\eta,
\end{equation}
\begin{equation}\label{2.4}
\xi* T*\eta\in \Cbr(G).
\end{equation}
Moreover, let us define $\f_{\xi,\eta}:G\rightarrow \C$ by
\begin{equation}\label{2.5}
\f_{\xi,\eta}(g)=\la u(g)\xi|\eta\ra\quad (g\in G).
\end{equation}
Then $\f_{\xi,\eta}\in \Cbr(G)$.
\end{lemma}

We now define a space of operators on $H$ whose coefficients are weakly almost periodic functions on $G$. Compare with \citep[Definition 2.3]{Jol}.

\begin{definition}\label{defwap}
A linear, bounded operator $T\in B(H)$ is $G$-\textit{weakly almost periodic} if, for all $\xi,\eta\in H$, the coefficient function $\xi*T*\eta$ belongs to $\WAP(G)$. The set of all such operators is denoted by $\wap(M,G)$.
\end{definition}

The proof of the following theorem is analoguous to that of \citep[Theorem 2.4]{Jol}, so we essentially omit it.

\begin{theorem}\label{defE}
Let $M\subset B(H)$ and $G\subset\Aut(M)$ be as above. The set $\wap(M,G)$ possesses the following properties:
\begin{enumerate}
\item [(a)] It is a unital, norm-closed operator system: it is a closed, selfadjoint subspace of $B(H)$ which contains $1$, and hence it is generated by its positive elements.
\item [(b)] For every $T\in\wap(M,G)$ and all $S_1,S_2\in u(G)'$, then one has $S_1TS_2\in \wap(M,G)$; in particular, $u(G)'\subset \wap(M,G)$, and $u(G)'$ is the set of elements $T\in\wap(M,G)$ such that the coefficient functions $\xi* T*\eta$ are all constant.
\item [(c)] The ideal $K(H)$ of compact operators on $H$ is contained in $\wap(M,G)$. 
\item [(d)] There exists a (unique) linear, bounded unital map $\E:\wap(M,G)\rightarrow B(H)$ which is characterized by
\begin{equation*}
\la \E(T)\xi|\eta\ra=\mk_G(\xi* T*\eta) \quad (\xi,\eta\in H),
\end{equation*}
and which possesses the following properties: 
\begin{enumerate}
\item [(i)] $\E$ is completely positive;
\item [(ii)] for every $T\in\wap(M,G)$, one has $\E(T)\in u(G)'$;
\item [(iii)] for all $T\in\wap(M,G), S_1,S_2\in u(G)'$, one has $\E(S_1TS_2)=S_1\E(T)S_2$;
\item [(iv)] for every $T\in\wap(M,G)$, $\E(T)$ belongs to $K_G(T)$, which is the weakly closed convex hull of the orbit $\{g(T)\colon g\in G\}$ in $B(H)$;
\item [(v)] for every $T\in \wap(M,G)$ and every $g\in G$, the operator $g(T)$ belongs to $\wap(M,G)$ et $\E(g(T))=\E(T)$.
\end{enumerate}
\end{enumerate}
\end{theorem}
Let us just sketch the proof of Property (d)(i): Since $\mk_G$ is a unital, positive functional, it is obvious that $\E$ is a positive map. If $n\geq 2$, we embed $G$ into $\Aut(M_n(M))$ by defining, for $g\in G$, the automorphism $g^{(n)}\in \Aut(M_n(M))$ by
\[
g^{(n)}(x_{ij})=(g(x_{ij}))\quad (g\in G,(x_{ij})\in M_n(M)).
\]
We observe that the range $G^{(n)}$ of the embedding of $G$ is a subgroup of $\Aut(M_n(M))$.
Then $M_n(\wap(M,G))$ is identified in a natural way to $\wap(M_n(M),G^{(n)})$, and the corresponding map is
\[
\E^{(n)}((T_{ij}))=(\E(T_{ij}))\quad ((T_{ij})\in M_n(\wap(M,G)).
\]
This proves that $\E$ is completely positive.
\hfill $\Box$

\section{Proof of the main theorem}

Let us state a slightly more precise statement than Theorem \ref{mainThm}.

\begin{theorem}
Let $M$ be a von Neumann and $G\subset \Aut(M)$ a group of automorphisms of $M$. Then $M$ is $G$-finite if and only if $M\subset \wap(M,G)$. Moreover, if it is the case, the restriction to $M$ of the map $\E$ in Theorem \ref{defE} is the conditional expectation $\E^G$ from $M$ onto $M^G$ in Theorem \ref{thm1.1}.
\end{theorem}
\begin{proof}
$(\Rightarrow)$ Let us fix $x\in M,\xi,\eta\in H$. We must show that $\xi* x*\eta\in \WAP(G)$. It suffices to prove that $C(x)\coloneqq \{(\xi* x*\eta)\cdot g\colon g\in G\}$ is weakly relatively compact in $\Cbr(G)$. For $g,h\in G$, let us compute:
\[
(\xi* x*\eta\cdot g)(h)=\xi* x*\eta(hg)=\la u(hg)xu(g^{-1}h^{-1})\xi|\eta\ra
=\la h(g(x))\xi|\eta\ra=\xi* g(x)*\eta(h)
\]
thus $C(x)=\{\xi* g(x)*\eta\colon g\in G\}$. By Corollary \ref{Gro2}, it suffices to prove that, for all sequences $(h_i),(g_j)\subset G$ for which the two double limits
\[
\ell_1=\lim_i(\lim_j(\xi* g_j(x)*\eta)(h_i))
\quad \textrm{et}\quad
\ell_2=\lim_j(\lim_i(\xi* g_j(x)*\eta)(h_i))
\]
exist, then $\ell_1=\ell_2$.

Up to passing to subsequences if necessary, we assume that $(g_j)$ converges to $\theta\in \ol G\subset B_*(M)$ and that $(h_i)$ converges to $\psi\in B_*(M)$, by Theorem \ref{thm1.1}(4).

Let $\omega_{\xi,\eta}\in M_*$ the normal linear form $y\mapsto \la y\xi|\eta\ra$. On the one hand, for every  $i$, one has
\[
\lim_j\la h_i(g_j(x))\xi|\eta\ra=\la h_i(\theta(x))\xi|\eta\ra=\la h_i\circ\theta,x\otimes \omega_{\xi,\eta}\ra=\la h_i,\theta(x)\otimes\omega_{\xi,\eta}\ra.
\]
Hence $\ell_1=\lim_i\la h_i,\theta(x)\otimes\omega_{\xi,\eta}\ra=\la\psi\circ\theta(x)\xi|\eta\ra$. 

On the other hand, one has $\ell_2=\lim_j \la \psi(g_j(x))\xi|\eta\ra=\la \psi\circ\theta(x)\xi|\eta\ra=\ell_1$.\\
$(\Leftarrow)$ Let us assume that $M\subset \wap(M,G)$. We are going to prove that $M$ and $G$ satisfy Condition (4) of Theorem \ref{thm1.1}, i.e. that every limit point $\psi$ in $B(M)$ of a sequence $(g_i)_{i\geq 1}\subset G$, with respect to the w$^*$-topology, belongs to $B_*(M)$. 

Let then $(g_i)_{i\geq 1}\subset G$ which $w^*$-converges to $\psi\in B(M)$. It is obvious that $\psi\in B(M)$ is unital and completely positive. Let us fix a sequence $(x_j)_{j\geq 1}\subset M_+$ such that $x_j\searrow 0$. We are going to show that $\lim_j\psi(x_j)=0$ with respect to the weak operator topology. In order to do that, let $\xi\in H$ and let us define $f_j:G\rightarrow \R^+$ by
\[
f_j(g)=\la g(x_j)\xi|\xi\ra=\xi*x_j*\xi(g)  \quad (g\in G).
\]
One has $f_j\in \WAP(G)$ for every $j$ by hypothesis. Set $A=\{f_j\colon j\geq 1\}\subset \Cbr(G)$, which is weakly relatively compact by Proposition \ref{relcpt}. 

Next, one has 
\[
\lim_j f_j(g_i)=\lim_j \la g_i(x_j)\xi|\xi\ra=0
\]
for every $i$ because $g_i\in \Aut(M)$ and $x_j\searrow 0$. Thus $\ell_1\coloneqq \lim_i(\lim_j f_j(g_i))=0$. Finally, 
\[
\lim_j(\lim_if_j(g_i))=\lim_j (\lim_i\la g_i(x_j)\xi|\xi\ra)=\lim_j\la \psi(x_j)\xi|\xi\ra\eqqcolon \ell_2.
\]
(Notice that the sequence $(\la\psi(x_j)\xi|\xi\ra)_j$ converges since it has its values in $\R_+$ and since it is decreasing.) By Proposition \ref{Gro1}, one has $\ell_2=0$, and then $\psi\in B_*(M)$.

Suppose at last that $M\subset \wap(M,G)$, and let $\E:\wap(M,G)\rightarrow u(G)'$ be the map of Theorem \ref{defE} and $\E^G:M\rightarrow M^G$ the conditional expectation of Theorem \ref{thm1.1}. 

By Properties (d)(ii) and (d)(iv) in Theorem \ref{defE}, for every $x\in M$, one has 
\[
\E(x)\in K_G(x)\cap u(G)'=\{\E^G(x)\},
\]
which proves that $\E(x)=\E^G(x)$. 
\end{proof}

\begin{remark}
Assume that $M\subset \wap(M,G)$. It is tempting to use the restriction to $M$ of the unital, cp map $\E$ in Theorem \ref{defE} to infer that $M$ is $G$-finite by condition (2) in Theorem \ref{thm1.1}. The problem is that, a priori, there is no reason for $\E$ that to be normal or faithfull on $M$. This explains why we had to prove another characterization of $G$-finiteness in the proof of the above theorem.
\end{remark}

\par\vspace{3mm}

\bibliographystyle{plain}
\bibliography{refGFinite}

\begin{thebibliography}{10}

\bibitem{HarpeMoy}
P.~de~la Harpe.
\newblock Moyennabilit\'e du groupe unitaire et propri\'et\'e {P} de {S}chwarz
  des alg\`ebres de von {N}eumann.
\newblock In P.~de~la Harpe, editor, {\em Alg\`ebres d'{O}p\'erateurs}, volume
  725 of {\em Lecture Notes in Mathematics}, pages 220--227. Springer-Verlag,
  Berlin, 1979.

\bibitem{EymardMoy}
P.~Eymard.
\newblock {\em Initiation \`a la th\'eorie des groupes moyennables}.
\newblock Lecture Notes in Math. 497, 1975.

\bibitem{Glas}
E.~Glasner.
\newblock {\em Ergodic {T}heory via {J}oinings}.
\newblock Amer. Math. Soc., Providence, RI, 2003.

\bibitem{Greenleaf}
F.P. Greenleaf.
\newblock {\em Invariant means on topological groups}.
\newblock Van Nostrand, New York, 1969.

\bibitem{Grothen}
A.~Grothendieck.
\newblock Crit\`eres de compacit\'e dans les espaces fonctionnels g\'en\'eraux.
\newblock {\em Amer. J. Math.}, 74:168--186, 1952.

\bibitem{haa}
U.~Haagerup.
\newblock Standard forms of von {N}eumann algebras.
\newblock {\em Math. Scand.}, 37:271--283, 1975.

\bibitem{Jol}
P.~Jolissaint.
\newblock Almost and weakly almost periodic functions on the unitary group of
  von {N}eumann algebras.
\newblock {\em J. of Operator Theory}, 87:101--124, 2022.

\bibitem{KS}
I.~Kov\'acs and J.~Sz\"ucs.
\newblock Ergodic type theorems in von {N}eumann algebras.
\newblock {\em Acta Sci. Math.}, 27:233--246, 1966.

\bibitem{Stoermer}
E.~St{\o}rmer.
\newblock Invariant states of von {N}eumann algebras.
\newblock {\em Math. Scand.}, 30:253--256, 1972.

\bibitem{Tak1}
M.~Takesaki.
\newblock {\em Theory of {O}perator {A}lgebras {I}}.
\newblock Springer Verlag, New York, 1979.

\bibitem{Tak2}
M.~Takesaki.
\newblock {\em Theory of {O}perator {A}lgebras {II}}.
\newblock Springer Verlag, Encyclopedia of Mathematical Sciences, Berlin, 2003.

\bibitem{Yeadon}
F.~Yeadon.
\newblock A new proof of the existence of a trace in a finite von {N}eumann
  algebras.
\newblock {\em Bull. Amer. Math. Soc.}, 77:257--260, 1971.

\end{thebibliography}

\end{document}